\documentclass{amsart}
\usepackage{amsthm,amssymb,amsrefs,a4wide}
\newtheorem{theorem}{Theorem}
\newtheorem{lemma}[theorem]{Lemma}
 
\newtheorem{remark}[theorem]{Remark}

\newcommand{\Ann}[2]{\mathrm{Ann}_{#1}({#2})}
\newcommand{\End}[2]{\mathrm{End}_{#1}({#2})}
\newcommand{\rann}[2]{\mathrm{r.ann}_{#1}({#2})}
\newcommand{\lann}[2]{\mathrm{l.ann}_{#1}({#2})}
\newcommand{\ZZ}{\mathbb{Z}}
\title{A counterexample for a problem on quasi Baer modules}
\subjclass[2000]{}
\author{Christian Lomp}
\address{Department of Mathematics of the Faculty of Science and  Center of Mathematics, University of Porto, Rua Campo Alegre, 687, 4169-007 Porto, Portugal}
\email{clomp@fc.up.pt}
\thanks{The author is a member of CMUP (UID/MAT/00144/2013), which is funded by FCT with national (MEC) and European structural funds (FEDER), under the partnership agreement PT2020.}
\begin{document}
\maketitle

\begin{abstract}
In this note we provide a counterexample to two questions on quasi-Baer modules raised recently by Lee and Rizvi in \cite{LeeRizvi}.
\end{abstract}

\section{Introduction}
Baer rings have been introduced by Kaplansky in \cite{Kaplansky}. In representing finite dimensional algebras as twisted matrix unit semigroup algebras, W.E. Clark introduced quasi-Baer rings in \cite{Clark}. Later, Rizvi and Roman generalized the setting of quasi-Baer rings to modules in \cite{RizviRoman}. In \cite{LeeRizvi}, 
Lee and Rizvi asked whether a quasi-Baer module is always quasi-retractable and whether a q-local-retractable module is local-retractable. For a unital right $R$-module $M$ over an associative unital ring $R$ and a subset $I$ of the the endomorphism ring $S=\mathrm{End}(M_R)$ we set $\Ann{M}{I}=\bigcap_{f\in I} \mathrm{Ker}(f)$.  Following \cite{LeeRizvi} $M$ is called a {\it quasi-Baer} module if
$\Ann{M}{I}$ is  direct summand of $M$, for any $2$-sided ideal $I$ of $S$. By \cite{LeeRizvi}*{Theorem 2.15}, a module $M$ is quasi-Baer if and only if $S$ is quasi-Baer and $M$ is {\it $q$-local-retractable}, where the latter means that $\Ann{M}{I} = \rann{S}{I}M$ for any 2-sided ideal $I$ of $S$ and $r.ann_S(I)$ is the right annihilator of $I$ in $S$. Related to the notion of $q$-local-retractability, Lee and Rizvi defined 
a module $M$ to be  {\it local-retractable} if  $\Ann{M}{L}=\rann{S}{L}M$, for any left ideal $L$ of $S$. Following \cite{RizviRoman1}*{Definition 2.3},  a module  $M$ is called  {\it quasi-retractable} if  $\rann{S}{L}=0$ implies $\Ann{M}{L}=0$, for any left ideal $L$ of $S$. 
The following questions have been raised at the end of \cite{LeeRizvi}:
\begin{enumerate}
\item Is a quasi-Baer module always quasi-retractable?
\item Is a q-local-retractable module local-retractable?
\end{enumerate}
We will answer these questions in the negative.  Note that any quasi-Baer module is q-local-retractable by \cite{LeeRizvi}*{Theorem 2.15} and any local-retractable module is quasi-retractable by definition. Hence a negative answer to the first question will also answer the second question in the negative.  We also note, that the notion of being quasi-Baer only involves 2-sided ideals of the endomorphism ring. Therefore, any module $M$ whose endomorphism ring $S$ is a simple ring is a quasi-Baer module. Moreover, if $S$ is a domain, then $\rann{S}{I}=0$ for any non-zero subset $I$ of $S$. Hence a module $M$ with $S$ being a domain is local-retractable if and only if it is quasi-retractable if and only if any non-zero endomorphism of $S$ is injective, because $\Ann{M}{Sf}=\mathrm{Ker}(f)$ for any $f\in S$. This means that any module $M$ whose endomorphism ring $S$ is a simple domain and that admits an endomorphism which is not injective yields a counter example to both of the questions.

For any ring $S$ and $S$-$S$--bimodule $_SN_S$,  consider the generalized matrix ring
$R=\left(\begin{array}{cc} S & N \\ N & S\end{array}\right)$ with multiplication given by 
$$ \left(\begin{array}{cc} a & x \\ y & b\end{array}\right)\cdot 
 \left(\begin{array}{cc} a' & x' \\ y' & b'\end{array}\right)
=
 \left(\begin{array}{cc} aa' & ax'+xb'  \\ ya'+by' & bb'\end{array}\right)$$
 for all $a,a',b,b' \in S$ and $x,x',y,y'\in N$. The ring $R$ can be seen as the matrix ring of the Morita context $(S,N,N,S)$ with zero multiplication $N\times N \rightarrow S$. Let $M=N\oplus S$. Then $M$ is a right $R$-module by the following action:
$$ (n,s) \cdot 
\left(\begin{array}{cc} a & x \\ y & b\end{array}\right) = (na+sy, sb).$$

Furthermore, $\End{R}{M} \simeq S$, because for any right $R$-linear endomorphism $f:M\rightarrow M$ with $f(0,1)=(n',t)$, we have that
$$(0,0)=f(0,0) = f \left( (0,1)\cdot \left(\begin{array}{cc} 1 & 0 \\ 0 & 0\end{array}\right)\right) = (n',t)\cdot  \left(\begin{array}{cc} 1 & 0 \\ 0 & 0\end{array}\right)
= (n',0).$$ Hence $n'=0$, i.e. $f(0,1)=(0,t)$ and therefore for all $(n,s)\in M$: 
$$f(n,s)=f\left((0,1)\cdot \left(\begin{array}{cc} 0 & 0 \\ n & s\end{array}\right)\right) = (0,t)\cdot  \left(\begin{array}{cc} 0 & 0 \\ n & s\end{array}\right)
= (tn,ts).$$
Thus $f$ is determined by the left multiplication of $t\in S$ on $M=N\oplus S$. Now it is easy to check that the map $S\rightarrow \End{R}{M}$ sending $t$ to  $\lambda_t\in \End{R}{M}$ with $\lambda_t(n,s) = (tn,ts)$, for any $(n,s)\in M$, is an isomorphism of rings.
This means that the structure of $M$ as left $\End{R}{M}$-module is the same as the  left $S$-module structure of $M$.

\begin{lemma}\label{lemma1} Let $N$ be an $S$-bimodule, $R=\left(\begin{array}{cc} S & N \\ N & S\end{array}\right)$  and $M=N\oplus S$ be as above. 
\begin{enumerate}
\item If $S$ is a simple ring, then $M$ is a quasi-Baer right $R$-module.
\item If $S$ is a domain, then $M$ is quasi-retractable if and only if $M$ is local-retractable if and only if $N$ is a torsionfree left $S$-module.
\item If $S$ is a simple domain and $N$ is not torsionfree as left $S$-module, then $M$ is a quasi-Baer right $R$-module that is neither quasi-retractable nor local-retractable.
\end{enumerate}
\end{lemma}

\begin{proof}
(1) Clearly if $S$ is a simple ring, then the only 2-sided ideals of $S\simeq \End{R}{M}$ are ${0}$ and $\End{R}{M}$, whose left annihilators are direct summands of $M$.

(2) Any local-retractable module is quasi-retractable. Suppose $S$ is a domain and $M$ is quasi-retractable. Let $0\neq s\in S$. Since $S$ is a domain, $\rann{S}{L}=0$, for $L=Ss$. Hence, as $M$ is quasi-retractable, $ 0=\Ann{M}{L}=\lann{N}{L}\oplus\lann{S}{L}$, i.e. $sn\neq 0$ for all $0\neq n\in N$. This shows that $N$ is torsionfree as left $S$-module.  If $S$ is a domain and $N$ is a torsionfree left $S$-module, then for any non-zero left ideal $L$ of $S$, $s\in S$ and $n\in N$: $L\cdot (n,s)=(Ln,Ls)\neq 0$, i.e. $\Ann{M}{L}=0$ if $L\neq 0$, which shows that $M$ is local-retractable.

(3) follows from (1) and (2).
\end{proof}

\begin{theorem} For any simple Noetherian domain $D$ with ring of fraction $Q$, such that $Q\neq D$, there exist a ring $R$ and a right $R$-module $M$ such that $\mathrm{End}({M_R})\simeq D$ and $M$ is a quasi-Baer right $R$-module, hence q-local-retractable, but neither quasi-retractable nor local-retractable.
\end{theorem}
\begin{proof}
The bimodule $N=Q/D$ and ring $S=D$ satisfy the conditions of Lemma \ref{lemma1}(3), i.e. $S$ is a simple domain and $N$ is not torsionfree.
\end{proof}
In order to obtain a concrete example, one might choose $D=A_n(k)$, the $n$-th Weyl algebra over a field $k$ of characteristic zero, or more generally $D=K[x;d]$ with $d$ a non-inner derivation of $K$ and $K$ a $d$-simple right Noetherian domain (see \cite{Lam}*{Theorem 3.15}).

\begin{remark} The construction in Lemma \ref{lemma1} is a special case of a more general construction. Let $G$ be a finite group with neutral element $e$. A unital associative $G$-graded algebra is an algebra $A$ with decomposition $A=\bigoplus_{g\in G} A_g$ (as $\ZZ$-module), such that $A_gA_h\subseteq A_{gh}$. The dual $\ZZ[G]^*$ of the integral group ring of $G$ acts on $A$ as follows. For all $f\in \ZZ[G]^*$ and $a=\sum_{g\in G} a_g$ one sets $f\cdot a = \sum_{g\in G} f(g)a_g$. Let $\{p_g \mid g\in G\}$ denote the dual basis of $\ZZ[G]^*$. The smash product of $A$ and $\ZZ[G]^*$ is defined to be the algebra $A\# \ZZ[G]^* = A\otimes_{\ZZ}  \ZZ[G]^*$ with product defined as $(a\# p_g)(b \# p_h) = a b_{gh^{-1}} \# p_h$ for all $a\#p_g, b\#p_h \in A\# \ZZ[G]^*$ (see \cite{CohenMontgomery}*{pp 241}). The identity element of $A\# \ZZ[G]^*$ is $1\# \epsilon$, where $\epsilon$ is the function defined by $\epsilon(g)=1$, for all $g\in G$. The algebra $A$ is a left $A\# \ZZ[G]^*$-module via $(a\# f) \cdot b = a (f\cdot b)$, for all $a\# f \in A\#\ZZ[G]^*$ and $b\in A$. Writing endomorphisms opposite to scalars we note that for any $z\in A_e$, the right multiplication $\rho_z: A\rightarrow A$ defined as $(b)\rho_z=bz$, for all $b\in A$ is left $A\# \ZZ[G]^*$-linear, because for any $f\in \ZZ[G]^*$, homogeneous $b\in A_g$ and $a \in A$ we have
$ ((a\# f)\cdot b)\rho_z = a f(g) b z = af(g) (b)\rho_z = (a\#f) \cdot (b)\rho_z,$ since $(b)\rho_z=bz\in A_g$.
On the other hand  let  $\psi:A\rightarrow A$ be any left $A\# \ZZ[G]^*$-linear map and set $z=(1)\psi$,  the image of $1$ under $\psi$. Then $z\in A_e$, because for any $g\in G\setminus \{e\}$ we have 
$p_g\cdot z = (1\#p_g)\cdot (1)\psi = (p_g\cdot 1)\psi = 0,$ as $1\in A_e$. Thus $z\in A_e$ and for any $b\in A$: $(b)\psi = ((b\# p_e)\cdot 1)\psi = (b\#p_e)\cdot z = bz = (b)\rho_z.$
Hence we have a bijective correspondence between elements of $A_e$ and left $A\# \ZZ[G]^*$-linear endomorphisms of $A$, i.e. $\rho: A_e \rightarrow \mathrm{End}({_{A\#\ZZ[G]^*}{A}})$ with $\rho(z)=\rho_z$ is an isomorphism of $\ZZ$-modules. Since $(b)\rho_{zz'}=bzz' = \left((b)\rho_{z}\right)\rho_{z'} = (b) (\rho_{z}\circ\rho_{z'})$ and $\rho_1 = \mathrm{id}_A$  we see that $\rho$ is an isomorphism of rings (see also \cite{Lomp}). To summarize: any associative unital $G$-graded algebra $A$ is a left module over $R=A\# \ZZ[G]^*$, such that $\mathrm{End}(_RA)$ is isomorphic to the zero component $A_e$. 

In the case of Lemma \ref{lemma1} consider  $A=S\oplus N$ as the trivial extension of $S$ by the $S$-bimodule $N$, i.e. $(a,x)(b,y)=(ab,ay+xb)$. Then $A$ is graded by the group $G=\{0,1\}$ of order $2$ with $A_0=S$ and $A_1=N$. It is not difficult to show that the  smash product $A\#(\ZZ G)^*$ is anti-isomorphic to the ring $R=\left(\begin{array}{cc} S & N \\ N &S\end{array}\right)$ and that the left action of $A\#(\ZZ G)^*$ on $A$ corresponds to the right action of $R$ on $M=N\oplus S$.
\end{remark}

\end{document}